\makeatletter

\def\eqalign#1{\,\vcenter{\openup\jot\m@th
  \ialign{\strut\hfil$\displaystyle{##}$&$\displaystyle{{}##}$\hfil
      \crcr#1\crcr}}\,}
\def\eqalignno#1{\displ@y \tabskip\@centering
  \halign to\displaywidth{\hfil$\displaystyle{##}$\tabskip\z@skip
    &$\displaystyle{{}##}$\hfil\tabskip\@centering
    &\llap{$##$}\tabskip\z@skip\crcr
    #1\crcr}}
\def\leqalignno#1{\displ@y \tabskip\@centering
  \halign to\displaywidth{\hfil$\displaystyle{##}$\tabskip\z@skip
    &$\displaystyle{{}##}$\hfil\tabskip\@centering
    &\kern-\displaywidth\rlap{$##$}\tabskip\displaywidth\crcr
    #1\crcr}}

\makeatother

\newdimen\pixel \pixel=.00333333 true in

\documentclass[11pt]{article}
\usepackage{amsmath}
\usepackage{amssymb}
\usepackage{epsfig}
\usepackage{fullpage}
\usepackage{color}
\usepackage{tikz}
\definecolor{light-gray}{gray}{0.3}

\makeatletter
\def\bigpar{\bigbreak\@afterindentfalse\@afterheading\ignorespaces}
\def\medpar{\medbreak\@afterindentfalse\@afterheading\ignorespaces}
\def\smallpar{\smallbreak\@afterindentfalse\@afterheading\ignorespaces}

\newlength{\saveindent}
\newenvironment{proof}%
      {\bigpar{\bf Proof:}\ 
             \setlength{\saveindent}{\parindent} 
                       \ignorespaces}%
      {\stopproof\ignorespaces\bigbreak \setlength{\parindent}{\saveindent}}

      {\bigpar{\bf Proof:}%
             \setlength{\saveindent}{\parindent} 
                       \,\ignorespaces}%
      {\ignorespaces\bigbreak \setlength{\parindent}{\saveindent}}

      {\bigpar{\bf Proof:}\ %
             \setlength{\saveindent}{\parindent} 
                       \ignorespaces}%
      {\ignorespaces\bigbreak \setlength{\parindent}{\saveindent}}

\newenvironment{proofof}[1]%
      {\bigpar{\bf#1:}\ %
             \setlength{\saveindent}{\parindent} 
                       \ignorespaces}%
      {\stopproof\ignorespaces\bigbreak \setlength{\parindent}{\saveindent}}

      {\smallpar{\bf Remark:}\ 
                       \ignorespaces}%
      {\stopproof\ignorespaces\medbreak \setlength{\parindent}{\saveindent}}

\newenvironment{remark*}%
      {\smallpar{\bf Remark:}\ 
                       \ignorespaces}%
      {\ignorespaces\medbreak \setlength{\parindent}{\saveindent}}

      {\smallpar{\bf Remarks:}\ 
                       \ignorespaces}%
      {\stopproof\ignorespaces\medbreak \setlength{\parindent}{\saveindent}}

\newenvironment{remarks*}%
      {\smallpar{\bf Remarks:}\ 
                       \ignorespaces}%
      {\ignorespaces\medbreak \setlength{\parindent}{\saveindent}}

      {\smallpar{\bf Remark:}\ %
                       \ignorespaces}%
      {\stopproof\ignorespaces\medbreak \setlength{\parindent}{\saveindent}}

      {\smallpar{\bf Remarks:}\ %
                       \ignorespaces}%
      {\stopproof\ignorespaces\medbreak \setlength{\parindent}{\saveindent}}
\makeatother

\newtheorem{theorem}{Theorem}
\newtheorem{lemma}[theorem]{Lemma}

\newtheorem{example}{Example}

\def\begex{\begin{example}\parindent=0pt \rm}
\def\endex{\end{example}}
\def\square{\vbox{\hrule height.2pt\hbox{\vrule width.2pt height5pt \kern5pt
                                   \vrule width.2pt} \hrule height.2pt}}
\def\stopproof{\hfill \square \smallskip}

\def \htt {{\hat q}}

\def \ap {{a_*}}
\def \bp {{b_*}}
\def \app {{a_{**}}}
\def \aux {{auxiliary }}

\def\rtr {{Random-to-Random }}
\def\ccr {{Card-Cyclic-to-Random }}

\newcommand{\1}{\mbox{\bf 1}}

\def \siigma {{\hat \sigma}}

\def \U {{ \cal U}}

\def\ofrac#1#2{{\textstyle{#1 \over #2}}}
\def\sfrac#1#2{{\textstyle{#1 \over #2}}}

\def\half{{\textstyle{1\over2}}}

\def \r {{\bf R}}

\def \r {{ \cal F}}

\def \ap {\alpha'}
\def \bp {\beta'}
\def\diam{{\rm diam}}

\def\r|{{\Bigr\vert}}
\def\l|{{\Bigl\vert}}

\def \R {{\bf R}}

\def\phi {\Phi}

\def\e{\epsilon}

\def\varepsilon{\mathchar"122 }

\def \chi {{\mathbf 1}}

\def\Gamma {{}}

\def\P {{ \mathbb P}}

\def\e {{ \mathbb {E}}}
\def\E {{ \mathbb {E}}}

\def\Sq{{\cal S}_q}
\def\Sq-{{\cal S}_{q-1}}

\def\f {{\cal F}}

\def\given {{\,|\,}}

\newcommand{\lab}{\label}

\newcommand{\be}{\begin{eqnarray}}
\newcommand{\ee}{\end{eqnarray}}
\newcommand{\eps}{{\mbox{$\epsilon$}}}

\begin{document}
\title{Mixing time of the \ccr shuffle}
\author{
{\sc Ben Morris}\thanks{Department of Mathematics,
University of California, Davis.
Email:
{\tt morris@math.ucdavis.edu}.
Research partially supported by
NSF grant DMS-1007739.}
\and
{\sc Weiyang Ning}\thanks{Department of Mathematics,
University of Washington.
Email:
{\tt ningw@u.washington.edu}.
}
\and
{\sc Yuval Peres}\thanks{
Microsoft Research.
Email:
{\tt peres@microsoft.com}.
}  }
\date{}
\maketitle
\begin{abstract}
\noindent
The \ccr shuffle on $n$ cards is defined as follows:
at time $t$ remove the card with label $t$ mod $n$ and randomly reinsert it back into the deck.
Pinsky \cite{P} introduced this shuffle and asked how many steps are needed to mix the deck. He showed $n$ steps do not suffice. 
Here we show that the mixing time is on the order of $\Theta(n \log n)$. 
\end{abstract}
Key words: Markov chain, mixing time.
\setcounter{page}{1}

\section{Introduction} \lab{intro}
In many Markov chains, such as Glauber Dynamics for the Ising model,
the state space is a set of configurations, and at each step a
location is chosen and updated.
An important general question about such chains
is what happens when we move from the world of {\it random updates},
where at each step a location is chosen at random
and updated, to {\it systematic scan},
when the updates are done in a more
deterministic fashion (see e.g., \cite{DR}). On the one hand, systematic scan is ``less random'',
so one might expect that the mixing
time is larger. On the other hand,
systematic scan can update $n$ sites in $n$ steps,
whereas with
random updates $n \log n$
steps are required by the coupon collector problem, so one might
expect systematic scan to have a smaller mixing time.

This question has been
investigated in the context of the random transpositions shuffle.
In this shuffle, at each step a pair of cards is chosen uniformly at random and interchanged.
In a classical result of Diaconis and Shahshahani \cite{DS}, the mixing time of the random transposition shuffle is shown to be asymptotically ${1\over 2}n \log n$. Mironov \cite{M}, Mossel, Peres and Sinclair \cite{MPS}, and Saloff-Coste and Zuniga \cite{SZ1} analyzed the Cyclic-to-Random shuffle, which is a systematic scan version of the random transposition shuffle:
at step $t$ the card in position $t$ mod $n$ is interchanged with a randomly chosen card.
They found that the mixing time for this chain is still on the order of $n \log n$.

In the present paper, we study a systematic scan version of the \rtr insertion shuffle.
In the \rtr insertion shuffle, at each step a card is chosen uniformly at random and then inserted in a uniform random position.
It was shown in \cite{DS1}, \cite{U} and \cite{S} that the mixing time of this shuffle is on the order of $n\log n$.
Pinsky \cite{P} introduced the following model, called the \ccr shuffle:
at time $t$ remove the card with the label $t$ mod $n$ and insert it in a uniform random position.
It is not obvious that the mixing time is greater than $n$: after $n$ steps the location of each card has been randomized,
so one might expect the whole deck to be close to uniform at time $n$. However, Pinsky showed that the mixing time is indeed greater than $n$,
since the total variation distance to stationarity at this time converges to $1$ as $n$ goes to infinity.
We show that in fact the mixing time is on the order of $n \log n$.
To prove the lower bound we
introduce the concept of a barrier
between two parts of the deck that moves along with the cards as the shuffling
is performed. Then we show that the trajectory of this barrier
can be well-approximated by a deterministic function $f$ satisfying
\be
\label{ma3}
f(x) = \int_{x-1}^x f(s) ds,
\ee and
we relate the mixing rate of the chain to the rate at which $f$
converges to a constant.
To prove the upper bound, we use the path
coupling method of Bubley and Dyer \cite{BD}.

\section{Statement of main results}
  Let $X_t$
be a Markov chain on a finite state space $V$
that converges to the uniform distribution.
For probability measures $\mu$ and $\nu$ on $V$,
define the {\it total variation distance}
$|| \mu - \nu || = \half \sum_{x \in V} |\mu(x) - \nu(x) |$,
and
define the $\epsilon$-mixing time
\[
\label{mixingtime}
t_{\rm mix}(\epsilon) = \min \{t:   || \P_x(X_t =  \cdot) - \U ||  \leq \epsilon \mbox{ for all $x \in V$}\} \,,
\]
where $\U$ denotes the uniform distribution on $V$.

Recall that in the \ccr shuffle,
at time $t$ we remove the card with label $t \bmod n$ and then reinsert it into a
uniform random location.

Define a round to be $n$ consecutive such shuffles. Note that the Markov chain that performs a round of the \ccr shuffle
at each step is time-homogeneous
with a doubly-stochastic transition matrix,
irreducible and aperiodic, hence converges to the uniform stationary distribution.
It follows that the \ccr shuffle converges to uniform as well.
Our main results show that the mixing time is on the order of $\log n$ rounds.

\begin{theorem}\label{lowerbound}
There exists $c_0$ such that for any $c<c_0$ and $0<\eps<1$, when $n$ is sufficiently large, we have
$$t_{\rm mix}(\eps)\geq cn\log n.$$ Here $c_0={1\over 2+2a}$ where $a$ is the smallest positive solution of equations $b=e^{a}\sin b$ and $a=e^{a}\cos b-1$. Numerically $c_0=0.161875162...$.
\end{theorem}

\begin{theorem}\label{upperbound}
For any $\epsilon > 0$ and $n\geq 4$, we have
$$t_{\rm mix}(\eps)\leq C(n\log n - 2n \log\epsilon),$$ where $$C={1\over \log2-\log(e-1)}=6.58664655...$$
\end{theorem}

\noindent{{\bf Remark.}}
Theorem \ref{lowerbound} and Theorem \ref{upperbound}
together establish that the \ccr shuffle
has a pre-cutoff in total variation distance. It is an interesting open problem to determine if cutoff occurs in this shuffle. For reference on cutoff phenomenon, see \cite[Chapter 18]{LPW}.

\noindent{{\bf Remark.}}
One can also consider a simpler shuffle where at time $t$, the card in position $t$ mod $n$ is removed and inserted in a uniform random position.
Call this Position-Cyclic-to-Random insertion. The time-reversal of this chain can be obtained by at time $t$, picking a uniform 
card and inserting it to location $t$ mod $n$. Considering the length of the longest increasing subsequence shows $\Omega(n\log n)$ steps are needed to mix(Ross Pinsky, personal communication). A matching upper-bound of $O(n\log n)$ follows from the work of Saloff-Coste and Zuniga. See \cite[Theorem 4.8]{SZ1}.


Theorems \ref{lowerbound} and \ref{upperbound} 
will be proved in Sections \ref{lower} and \ref{upper}, respectively.

\section{Lower bound}
\label{lower}
\subsection{The barrier}
The key idea for the lower bound is to imagine a barrier between two parts of the deck, that moves along with the cards as the shuffling
is performed.
If a card is inserted into the gap that the barrier occupies, we use the convention that the card is inserted
on the same side of the barrier as it was in the previous step.
We illustrate this with the following example. Suppose there is a deck of 8 cards with a barrier between
cards 3 and 5. In the next step, card 7 is inserted between cards 3 and 5.
\[\begin{array}{ccccccccc}
2&1&3&|&5&4&6&8&7\\
2&1&3&|&7&5&4&6&8
\end{array}\]
Let $\{\sigma_t\}_{t=0}^{\infty}$ be a \ccr shuffle. We think of $\sigma_t(i)$ as the position
of card $i$ at time $t$, where the positions range from $1$ at the left to $n$ at the right.
Define the position of the barrier as the position of the card immediately to its left,
and throughout the present chapter, let $B_{t}$ be the position of the barrier at time $t$. Use the convention that $B_t=0$ if at time $t$ the barrier is to the left of all cards.
We will call the pair process $( \sigma_t, B_t)$
the {\it \aux process}.

Note that
the conditional probability that the card at time $t$ is inserted to the left of the barrier, given $B_t$,
is $\sfrac{1}{n}B_t$.
Since at
any time $t > n$, every card has been moved exactly once in the previous
$n$ steps, we have
$$B_t = \sum_{i = 1}^n
\1(\mbox{the card moved at time $t-i$ is inserted to the left of barrier}),$$
and hence
\be
\e(B_t) = \frac{1}{n} \sum_{i=1}^n \e(B_{t - i}).
\ee
Define
$g(t) = \e(\sfrac{1}{n}B_t)$. Then $g$ satisfies
the following {\it moving average} condition:
\be
g(t) = \frac{1}{n} \sum_{i=1}^n g(t - i),
\ee
for $t > n$.
We shall approximate $g(t)$ by
$f(t/n)$, where $f:\R \to [0,1]$ is a continuous function
satisfying (\ref{ma3}).
Our first lemma gives an example of such a function.
\begin{figure}[!ht]
\centering
\includegraphics[width=10cm]{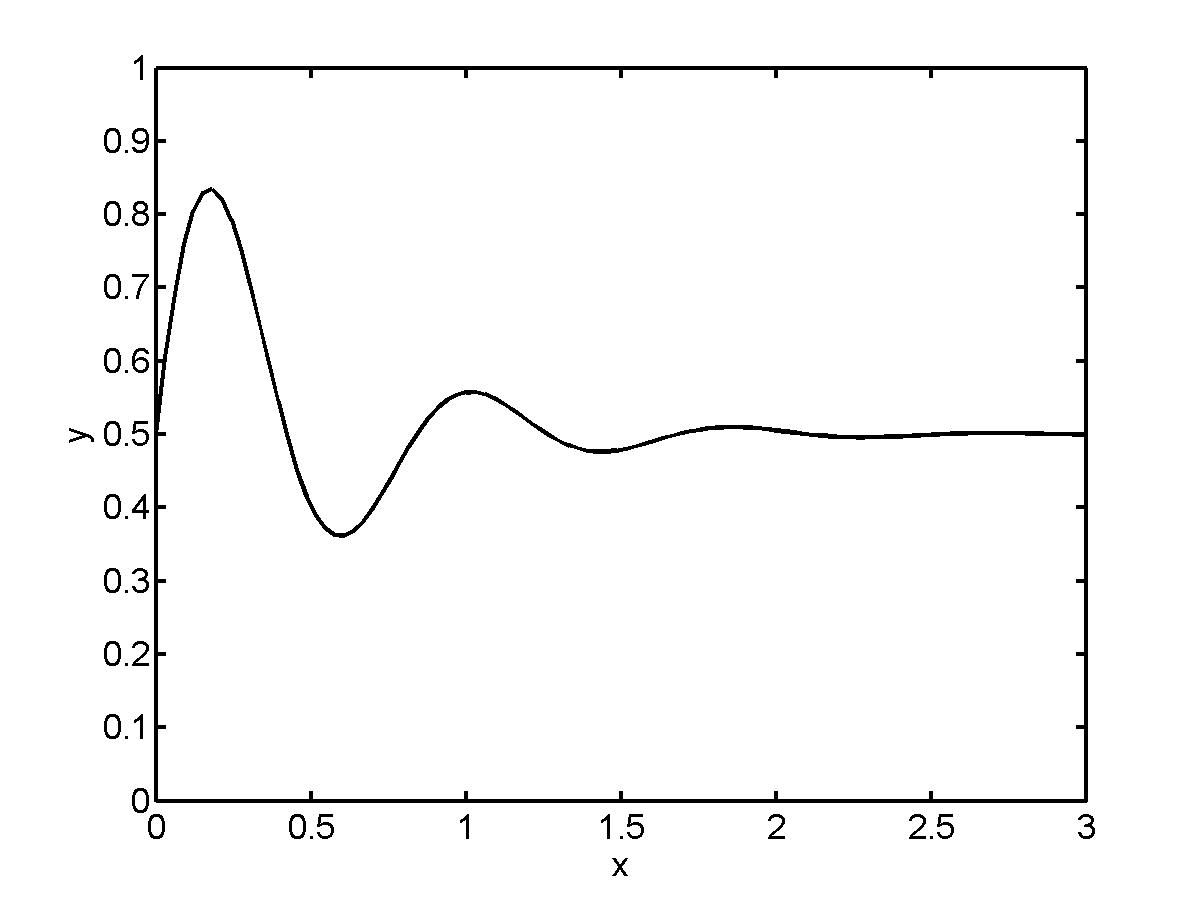}
\caption{Graph of f(s)}
\end{figure}
\begin{lemma}
\label{exists}
There exists $a>0$ and $b > 2\pi$
such that $f(x)= \half + \half e^{-ax}\sin(bx)$ satisfies
\be\label{relation1}f'(x)= f(x) - f(x-1).\ee
Moreover,
\be\label{relation2}f(x)=\int_{x-1}^xf(s)ds,\ee
for all $x$.
\end{lemma}
\begin{proof}
Since properties (\ref{relation1}) and
(\ref{relation2}) are preserved under shifting and scaling,
it is enough to show that they apply
to
$h(x) = e^{-ax} \sin(bx)$,
for suitable $a$ and $b$.

First, we show that
for suitable choice of $a$ and $b$ we have
$h'(x) = h(x) - h(x-1)$.
By the product rule, \be\label{equation1}h'(x)=-ae^{-ax}\sin(bx)+be^{-ax}\cos(bx),\ee and
a calculation shows that
\be\label{equation2}h(x)-h(x-1)=(1-e^{a}\cos b)e^{-ax}\sin(bx)+(e^{a}\sin b) e^{-ax}\cos(bx).\ee
The quantities (\ref{equation1}) and (\ref{equation2}) are equal if
$b=e^{a}\sin b$ and $-a=1-e^{a}\cos b$. Solving for $a$ in the first equation gives
$$a=\log{b\over \sin b},$$ and substituting this into the second one gives
$$\log{\sin b\over b}=1-{b\cos b\over \sin b}.$$
By the intermediate value theorem,
this equation has a solution with $b$
in the interval $[2 \pi + \sfrac{\pi}{4}, 2 \pi + \sfrac{\pi}{2}]$, since when $b = 2\pi + \sfrac{\pi}{4}$
the right-hand side is smaller than the left-hand side, but when
$b = 2\pi + \sfrac{\pi}{2}$ the right-hand side is larger.
Furthermore, since
$\sin b < b$ when $b > 0$, we have $a = \log{ b \over\sin b} > 0$.
(Numerical approximation gives the solution as
$b=7.4615...$ and $a=2.0888...$.)

Next we claim that since $h'(x) = h(x) - h(x-1)$, we must have $h(x) = \int_{x-1}^x h(s) ds$.
To see this,
define $\htt(x)=\int_{x-1}^xh(s)ds$ and note that $\htt'(x)=h'(x).$ This implies that $h(x)-\htt(x)=C$ for a constant $C$.
But since $a>0$, we have $h(x)\rightarrow0$ as $x\rightarrow\infty$. Consequently
$\htt(x)\rightarrow0$ as $x\rightarrow\infty$ by the definition of $\htt$, and so $C=0$.
\end{proof}
Recall that $g(t) = \e( \sfrac{1}{n} B_t)$, where $B_t$ is the position of the barrier at time
$t$. A key part of our proof will be to show that $g$ closely follows the continuous
function $f$ of Lemma \ref{exists}. However, in order for this to be the case we must start
with a permutation chosen from a certain probability distribution.
It is most convenient to describe this starting permutation as being generated
in the first $n$ time steps, which we call the {\it startup round.}
In the startup round, we begin with only a barrier.
At time $t$, for
$1 \leq t \leq n$, we put card $t$ to the left of the barrier with probability
$f(\sfrac{t}{n})$. The location among the already existing cards in the left (right) side of the barrier is arbitrary. We must modify the definition of $g$ to handle the startup round. Define
$g: \{1, 2, \dots\} \to \R$ by
\[
g(t) = \left\{\begin{array}{ll}
f(\sfrac{t}{n}) & \mbox{if $1 \leq t \leq n$;} \\
\e(\sfrac{1}{n} B_t) & \mbox{otherwise.} \\
\end{array}
\right.
\]
Thus $g$ satisfies the moving average condition, and, because
of the insertion probabilities used in the startup round,
$g$ matches $f$ for the first $n$ steps.
(That is,
$g(\cdot) =
f(\sfrac{\cdot}{n})$ on $\{1, \cdots, n\}$.)
As we show below,
this is enough to ensure that $g$ is well-approximated by $f$
for a number of rounds on the order of $\log n$.

\begin{lemma}\label{comparisonfg}
There exists a constant $C>0$ such that
$$\left|g(t)-f(\ofrac{t}{n}) \right|\leq \ofrac{C}{2n} e^{2 (t+1)/n},$$
for all $t > 0$.
\end{lemma}

\begin{proof}
First, note that if $t > n$ then
\begin{eqnarray*}
g(t+1) - g(t) &=& \ofrac{1}{n} \sum_{i=1}^n g(t+1-i) -
\ofrac{1}{n} \sum_{i=1}^n g(t-i) \\
&=& \ofrac{1}{n} \Bigl(g(t) - g(t-n)\Bigr).
\end{eqnarray*}
Rearranging terms gives
\be
\label{one}
g(t+1) = (1 + \ofrac{1}{n}) g(t) - \ofrac{1}{n} g(t-n).
\ee
Recall that $f(x) = \half + \half e^{-ax} \sin(bx)$ and $a > 0$.
Some calculus shows that the second derivative of $f$
is uniformly bounded on $[0, \infty)$. Hence
\begin{eqnarray*}
f(\ofrac{t+1}{n}) - f(\ofrac{t}{n}) &=& \ofrac{1}{n} f'(\ofrac{t}{n}) + O(\ofrac{1}{n^2}) \\
&=&
\ofrac{1}{n} \Bigl(f(\ofrac{t}{n}) -  f(\ofrac{t}{n} - 1) \Bigr)
+ O(\ofrac{1}{n^2}),
\end{eqnarray*}
where the first line follows from Taylor's theorem and the second line
follows from Lemma \ref{exists}. Rearranging terms gives
\be
\label{two}
f(\ofrac{t+1}{n}) =
(1 + \ofrac{1}{n})  f(\ofrac{t}{n}) -
\ofrac{1}{n}  f(\ofrac{t-n}{n})
+ O(\ofrac{1}{n^2}).
\ee
Combining (\ref{one}) and (\ref{two}) and using the triangle inequality gives
\be
\label{growing}
\left| g(t+1) - f(\ofrac{t+1}{n}) \right| \leq (1 + \ofrac{1}{n})
\left|g(t) - f(\ofrac{t}{n}) \right|
+ \ofrac{1}{n}
\left|g(t-n) - f(\ofrac{t-n}{n}) \right| + \ofrac{C}{n^2},
\ee
for a universal constant $C$. We claim that for all $t$ we have
\be
\label{almost}
\left| g(t) - f(\ofrac{t}{n}) \right| \leq \ofrac{C}{n^2}
\sum_{i=0}^t \left(1 + \ofrac{2}{n}\right)^i .
\ee
We prove this by induction. For the base case, note that
$g(t) = f(\ofrac{t}{n})$  for $t = 1, \dots, n$. Now if we suppose that (\ref{almost})
holds for $1,\dots, t$,
then the two absolute values on the right-hand side of
(\ref{growing}) can be bounded by $\ofrac{C}{n^2} \sum_{i=0}^t (1 + \ofrac{2}{n})^i$.
Hence
\begin{eqnarray*}
\left| g(t+1) - f(\ofrac{t+1}{n}) \right| &\leq& (1 + \ofrac{2}{n})
\Bigl[\ofrac{C}{n^2} \sum_{i=0}^t (1 + \ofrac{2}{n})^i \Bigr]
+ \ofrac{C}{n^2} \\
&=& \ofrac{C}{n^2} \sum_{i=0}^{t+1} (1 + \ofrac{2}{n})^i,
\end{eqnarray*}
which verifies (\ref{almost}) for $t+1$.
To finish the proof of the lemma, note that
\begin{eqnarray*}
\label{alf}
\ofrac{C}{n^2} \sum_{i=0}^t (1 + \ofrac{2}{n})^i  &=&
\ofrac{C}{n^2}  \,\, {
(1 + \ofrac{2}{n})^{t+1} - 1 \over \ofrac{2}{n}} \\
&\leq&  \ofrac{C}{2n} e^{2(t+1)/n}.
\end{eqnarray*}

\end{proof}

\subsection{Deviation estimates}
In the previous subsection we proved that the
expected barrier location
is well-approximated by a continuous function.
In the present subsection we show that the barrier
stays reasonably close to its expectation with high probability
when the number of rounds is on the order of $\log n$.

Define a {\it configuration} as a pair $(\sigma, b)$,
where $\sigma$ is a permutation and $b$ is a barrier location.
(Thus the state space of the \aux process is the set of
all configurations.)
We define the
{\it insertion distance} between two configurations as the minimum number of
cards we would need to remove and re-insert to get
from one configuration to the other. For example the insertion distance
between the two configurations below is $2$. (Move cards 4 and 7.)
\[\begin{array}{ccccccccc}
2&1&4&3&|&5&6&8&7\\
2&1&3&7&|&5&4&6&8
\end{array}\]
\begin{lemma}\label{deviationlemma}
Let $(\sigma^1_t, B^1_t)$ and $(\sigma^2_t, B^2_t)$ be
\aux processes, and define
$\siigma^i_t =
(\sigma^i_t, B^i_t)$ for $i=1,2$.
Let $d$ be the insertion distance between $\siigma^1_0$ and $\siigma^2_0$.
Then
$$|\E B^1_{t}-\E B^2_{t}|\leq d \left(1+ \ofrac{1}{n}\right)^t.$$
\end{lemma}
\begin{proof}
There is
a natural coupling of
$\sigma^1_t$ and $\sigma^2_t$ that we call {\it label coupling.}
In label coupling, at time $t$ we choose a label $X$ uniformly at random.
If $X = t \bmod n$, then we move card $t \bmod n$ to the leftmost position
in both processes. Otherwise, we insert card $t \bmod n$ to
the right of the card with label $X$ in both processes.

Suppose that $A = \{a_1, \dots, a_d\}$ is a minimal set of cards that
can be moved to get from
$\siigma^1_0$ to $\siigma^2_0$.
Note that under the label coupling,
only in the case when we move a card not in $A$ can the insertion distance be increased.
In such moves, if the card is put to the right of a card in
$A$, the insertion distance increases by $1$ and otherwise it stays the same.
Thus the expected insertion distance after one step is at most
\begin{eqnarray*}
(d+1){d \over n}+ d \,{n-d \over n}
=d \left(1+{1\over n}\right).
\end{eqnarray*}
Iterating this argument shows that the expected insertion distance after $t$ steps
is at most
$d \left(1+{1\over n}\right)^t$.
The lemma follows from this, since the barrier can move by at most one position
with each re-insertion.
\end{proof}
We are now ready to state the main lemma of this subsection.
\begin{lemma}\label{deviation}
Let $(\sigma_t, B_t)$ be an \aux process.
Fix $c >0$ and suppose $T$ satisfies $n < T \leq c n \log n$. Then
for any $x>0$ we have
$$\P(\left|\ofrac{1}{n}B_{T}-g(T) \right|
\geq x)
\leq 2\exp\left(- x^2 n^{1-2c}\right).$$
\end{lemma}
\begin{proof}
Fix $T$ with $n < T \leq c n \log n$.
Since $g(T) = \sfrac{1}{n} \e(B_T)$, it is enough to show that
for any $x > 0$ we have
$$\P(\left|B_{T}- \e(B_T) \right|\geq x)
\leq 2\exp\left(- x^2 n^{-(1+2c)}\right).$$
Let $\f_t$ be the sigma-field generated by the process up to time $t$,
and consider the
Doob martingale
\[
M_t := \e ( B_T \given \f_t).
\]
Applying Lemma \ref{deviationlemma} to the case of
two configurations that differ by one insertion gives
\[
| M_t - M_{t-1} | \leq \left( 1 + \ofrac{1}{n} \right)^{T - t},
\]
for $t$ with $1 \leq t \leq T$. Thus the Azuma-Hoeffding bound gives
\begin{eqnarray}
\P( | B_T - \e(B_T) | \geq x ) &=&
\P( | M_T - \e(M_T) | \geq x ) \nonumber\\
\label{sum}
&\leq& 2 \exp \left( {-x^2 \over 2\sum_{t=1}^T b_t^2 } \right),
\end{eqnarray}
where $b_t = \left(1 + \ofrac{1}{n} \right)^{T-t}$.
Let $r = \left(1 + \ofrac{1}{n} \right)^2$. The sum in (\ref{sum})
can be written as
\begin{eqnarray}
\sum_{i=0}^{T-1} r^i &=& {r^{T} - 1 \over r-1}\nonumber \\
\label{alf}
&\leq& \ofrac{n}{2} r^T,
\end{eqnarray}
since $r-1 = \sfrac{2}{n} + \sfrac{1}{n^2} \geq \sfrac{2}{n}$.
Since $T < c n \log n$, the quantity (\ref{alf}) is at most
\[
\ofrac{n}{2}
\left( 1 + \ofrac{1}{n} \right)^{2 c n \log n} \leq \half n^{1 + 2c}.
\]
Substituting this into (\ref{sum}) yields the lemma.
\end{proof}

\subsection{Proof of the Lower bound}
Recall that $f(s) = \half + \half e^{-as} \sin(bs)$,
for some $a = 2.0888...$ and $b = 7.4615...$.
The rough idea for the lower bound is as follows.
Note that if $c$ is sufficiently small and $s < c \log n$,
then the fluctuation of $f(s)$ between $s$ and $s+1$
is of higher order than $n^{-1/2}$.
Thus
in the
corresponding round of the \ccr shuffle,
there will be an interval of cards where the probability
of inserting to the left of the barrier is detectably high.
Before we give the proof, we recall Hoeffding's bounds in \cite{H}.
\begin{theorem}
Let $X_1, \dots, X_k$ be samples from a population of $0$'s and $1$'s,
and let $p = \e(X_1)$ be the proportion of $1$'s in the population.
Then
\begin{equation}
\label{hoeffding}
\P\left( \sum_{i=1}^k X_i - kp \geq \alpha\right) \leq e^{-2 \alpha^2/k}.
\end{equation}
The bound (\ref{hoeffding}) applies whether the sampling
is done with or without replacement.
\end{theorem}
\begin{proofof}{Proof of Theorem \ref{lowerbound}}
Let $c > 0$ be small enough so that
\begin{equation}
\label{cstuff}
c<{1\over 2+2a}.
\end{equation}
Fix $T$ with $n < T < c n \log n$ and let $x = T/n$.
Suppose that $\sin(bx) \leq 0$.
The case $\sin(bx) > 0$ is similar.
Since $b > 2\pi$, there exist
$x_1, x_2$ with $x-1 < x_1 < x_2 < x$, such that
\[
bx_1 = 2\pi k + \pi/4,\;\;\;\; \mbox{and}
\]
\[
bx_2 = 2\pi k + 3\pi/4, \;\;\;\;\;\;\;\;\;
\]
for an integer $k$. Note that for $s \in [x_1, x_2]$ we have
\begin{eqnarray}
f(s) &\geq& \half + \beta e^{-as}\nonumber \\
\label{fbound}
&\geq& \half + \beta n^{-ac},
\end{eqnarray}
where $\beta = \half \sin(\pi/4)$.
The second inequality holds because
$x \leq c \log n$.

Let $A$ be the event that
$|\sfrac{1}{n} B_t - f(t/n)| \leq {\beta \over 4} n^{-ac}$
for $t$ with $T - n < t \leq T$.
Note that since $T < c n \log n$, substituting
$T$ into the upper bound of Lemma \ref{comparisonfg} implies that if $t \leq T$ then
$\left | g(t) - f(t/n) \right| < B n^{2c - 1}$, for
a constant $B > 0$.
Since $2c-1 < -ac$ by (\ref{cstuff}),
for sufficiently large $n$ we have
\[
B n^{2c - 1} < \sfrac{\beta}{8} n^{-ac},
\]
and hence
$\left | g(t) - f(t/n) \right| < \sfrac{\beta}{8} n^{-ac}$
for $t \leq T$.
Hence
\begin{eqnarray}
\P(A^c) &\leq& \P( \left| \sfrac{1}{n} B_t - g(t) \right| > \sfrac{\beta}{8} n^{-ac}
\;\;\mbox{for some $t$ with $T-n < t \leq T$})\nonumber \\
&\leq& 2n \exp \Bigl(- \left[ \sfrac{\beta}{8} n^{-ac} \right]^2 n^{1-2c} \Bigr)\nonumber \\
&=&
\label{tz}
2n \exp \Bigl(- \sfrac{\beta^2}{64} n^{1 - 2c(a+1)} \Bigr),
\end{eqnarray}
where the second inequality follows from Lemma \ref{deviation} and
a union bound. Since $1  - 2c (a+1) > 0$ by (\ref{cstuff}),
the quantity (\ref{tz}), and hence $\P(A^c)$,
converges to $0$ as $n \to \infty$.

Let $I = \{ t \bmod n: nx_1 < t < nx_2 \}$ and $m = |I|$.
Since $x_2 - x_1 = \pi/2b$, there is a constant $\lambda$ such that
$m \geq \lambda n$ for sufficiently large $n$.
Let $N$ be the number of cards in $I$
(that is, cards whose label is in $I$)
placed to the left
of the barrier between times $nx_1$ and $nx_2$. Then $N$
is also the number of cards from $I$ to the left of the barrier at time $T$.
By (\ref{fbound}), on the event $A$ the insertion probabilities $\sfrac{B_t}{n}$
are bounded below by $\half + \sfrac{3 \beta}{4} n^{-ac}$ for
$t$ with $nx_1 < t < n x_2$.
Hence
the conditional distribution of $N$ given $A$ stochastically
dominates the Binomial($m, \half + \sfrac{3 \beta}{4} n^{-ac}$)
distribution. Thus Hoeffding's bounds give
\begin{eqnarray}
\P\left( N < \sfrac{m}{2} + \sfrac{\beta}{2} m n^{-ac} \given A \right)
&\leq& \exp\left({-2 \left( \sfrac{\beta}{4} m n^{-ac} \right)^2 \over m} \right)\nonumber \\
\label{gz}
&\leq& \exp \left(- \sfrac{\beta^2}{8} \lambda n^{1-2ac} \right),
\end{eqnarray}
where the second line follows from the fact that $m \geq \lambda n$. Since $1 - 2ac > 0$
by (\ref{cstuff}), the quantity (\ref{gz}) converges to $0$
as $n \to \infty$.

Now let $Y$ be the number of cards in $I$ having position less than
$\sfrac{n}{2} + \sfrac{\beta}{4} n^{1 - ac}$
at time $T$.
Since $f\left( T \over n \right) \leq \half$, we
have $B_T \leq
\sfrac{n}{2} + \sfrac{\beta}{4} n^{1 - ac}$ on the event $A$, and hence
\begin{eqnarray}
\P( Y \leq \sfrac{m}{2} + \sfrac{\beta}{2} m n^{-ac}) &\leq&
\label{eee}
\P( N \leq \sfrac{m}{2} + \sfrac{\beta}{2} m n^{-ac} \given A) + \P(A^c),
\end{eqnarray}
which converges to $0$
as $n \to \infty$.

To complete the proof, let $Y_u$ be the number
of cards in $I$ whose position is
less than
$\sfrac{n}{2} + \sfrac{\beta}{4} n^{1 - ac}$
in a uniform random permutation.

Hoeffding's bounds imply that
\begin{eqnarray}
\P(
Y_u > \sfrac{m}{2} + \sfrac{\beta}{2} m n^{-ac}
)
&\leq&
\exp\left( { -2 \left(\sfrac{\beta}{4} m n^{-ac} \right)^2 \over m} \right) \nonumber\\
&\leq&
\label{fff}
\exp \left( -  {\beta^2 \over 8} \lambda n^{1 - 2ac} \right),
\end{eqnarray}
for sufficiently large $n$.
Since $1 - 2ac > 0$, the quantity (\ref{fff}) converges to $0$ as $n \to \infty$. Combining this with (\ref{eee}), we conclude that $t_{\rm mix}(\eps)\geq cn\log n$ for large enough $n$.
\end{proofof}

\section{Upper bound}
\label{upper}
We use the path coupling technique introduced by Bubley and Dyer \cite{BD}.
Let $S_n$ be the permutation group and $G=(S_n, E)$,
where an edge exists between two permutations if and only if they differ by an adjacent transposition.
The path metric on $G$ is defined by
$$\rho(x,y)=\min\{{\rm length\,\, of\, \,\eta}: \eta \,\,{\rm is\,\,a\,\, path\, \,from\,\,} x\,\,{\rm to}\,\, y\}.$$
Define $$\diam(G)=\sup_{x,y}\rho(x,y).$$
The following theorem is from \cite{BD}. See also \cite[Chapter 14]{LPW}.

\begin{theorem}\label{pathcoupling}
Suppose that there exists $\alpha>0$ such that for each edge $\{x,y\}$ in $G$ there exists a coupling $(X_1, Y_1)$ of
the distributions $\P(x,\cdot)$ and $\P(y,\cdot)$ such that $$\E_{x,y}\rho(X_1,Y_1)\leq \rho(x,y)e^{-\alpha}.$$
Then $$t_{\rm mix}(\eps)\leq\frac{-\log\eps+\log(\diam(G))}{\alpha}.$$
\end{theorem}
For a permutation $x$, define $\sigma_t^x$ to be the
\ccr shuffle starting at $x$.
Our mixing time upper bound follows from the following lemma.
\begin{lemma}\label{onestep}
If permutations
$x$ and $y$ differ by an adjacent transposition and $n\geq 4$, there is a coupling of $\sigma_n^x$ and $\sigma_n^y$ such that
$$\E\rho(\sigma_n^x,\sigma_n^y)\leq e^{-\alpha},$$
where $\alpha = 2(\log 2 - \log(e-1))$.
\end{lemma}
\begin{proof}
There is another natural coupling of
two \ccr processes besides label coupling; we call this second
coupling {\it position coupling}. In position coupling,
the card is inserted into the same locations in both processes.
Now assume that for some $i < j$, the permutation $x$
can be obtained from $y$ by transposing the cards with label $i$ and $j$, as shown below.
In the diagram, the $k$th $X$ in the top row represents the same
card as the  $k$th $X$ in the bottom row.
\[\begin{array}{ccccccccc}
x:&X&X&X&i&j&X&X&X\\
y:&X&X&X&j&i&X&X&X
\end{array}\]
The coupling strategy is divided into $3$ stages,
corresponding to $t$ in
$\{1,\cdots, i-1\}$, $\{i,\cdots,j-1\}$, and $\{j,\cdots,n\}$ respectively.

Stage 1: moving cards $1, \dots, i-1$.
In this stage
use position coupling.
As is shown by diagram 1 below,
at the end of this stage
we still have two permutations that differ only by a transposition of $i$ and $j$.
However, there may have been some cards inserted between cards $i$ and $j$; we represent these cards with $a$'s.
\[\begin{array}{ccccccccc}
\sigma_{i-1}^x:&X&X&i&a&a&j&X&X\\
\sigma_{i-1}^y:&X&X&j&a&a&i&X&X
\end{array}\]
\begin{center}
diagram 1
\end{center}
Stage 2: moving cards $i, \dots, j-1$. In this stage we use label coupling.
At the end of this stage, some cards might have been
inserted into the group of $a$'s. We denote such cards with
$\ap$'s. In addition, some cards might have been inserted between
card $j$ and the first $X$ to the right of the card $j$. We represent them with $b$'s.
Diagram 2 shows a typical pair of permutations after stage 2.
\[\begin{array}{cccccccccccc}
\sigma_{j-1}^x:&X&X&a&\ap&a&\ap&j&b&b&X&X\\
\sigma_{j-1}^y:&X&X&j&b&b&a&\ap&a&\ap&X&X
\end{array}\]
\begin{center}
diagram 2
\end{center}
Stage 3:  moving cards $j, \dots, n$. Here  we use label coupling again.
Cards inserted into  the group of $a$'s and $\ap$'s are represented with $\app$'s,
and cards inserted into the group of $b$'s are represented with $\bp$s.
See diagram 3 below.
\[\begin{array}{cccccccccccccc}
\sigma_n^x:&X&X&a&\ap&\app&a&\app&\ap&b&\bp&b&X&X\\
\sigma_n^y:&X&X&b&\bp&b&a&\ap&\app&a&\app&\ap&X&X
\end{array}\]
\begin{center}
diagram 3
\end{center}
For $t \leq n$, let $A_t$ be the number of $a$'s, $\ap$'s and $\app$'s,
and let $B_t$ be the number of $b$'s and $\bp$'s, after card $t$ has been moved.
Note that $$\rho(\sigma_n^x, \sigma_n^y)\leq A_nB_n.$$
Thus we are left to estimate $\E (A_nB_n)$.

Initially we have $A_0=B_0=0$. Recall that in the first stage we use position coupling.
For $t\leq i-1$ we have $B_t=0$ and $A_t$ satisfies
$$\P(A_{t+1}=A_t|A_t)={n-A_t-1\over n},$$ and $$\P(A_{t+1}=A_t+1|A_t)={A_t+1\over n}.$$
This implies
\be\E(A_{t+1}+1|A_t)=(A_t+1)\left(1+{1\over n}\right).\ee
Hence
\be\label{firststep}\E A_{i-1}= \left(1+{1\over n}\right)^{i-1}-1.\ee

Recall that we use label coupling in the second stage.
For $i\leq t\leq j-1$, we have the following transition rule:
$$\P(A_{t+1}=A_t, B_{t+1}=B_t|A_t,B_t)={n-A_t-B_t-1\over n},$$ and $$\P(A_{t+1}=A_t+1, B_{t+1}=B_t|A_t,B_t)={A_t\over n},$$ and
$$\P(A_{t+1}=A_t, B_{t+1}=B_t+1|A_t,B_t)={B_t+1\over n}.$$ This implies
$$\E(A_{t+1}(B_{t+1}+1)|A_t, B_t)=A_t(B_t+1)\left(1+{2\over n}\right).$$
Recall that $B_t=0$ for all $t\leq i-1$. Thus we have
\be
\label{dagger}
\E A_{j-1}(B_{j-1}+1)=\E A_{i-1}\left(1+{2\over n}\right)^{j-i}.
\ee
Note that for $t$ with $i \leq t < j$ we have
\be\E(A_{t+1}|A_t)=A_t\left(1+{1\over n}\right).\ee
Thus $\E A_{j-1}=\E A_{i-1}(1+{1\over n})^{j-i}$. Combining this with (\ref{dagger}) and
(\ref{firststep}) gives
\be\label{secondstep}
\E A_{j-1}B_{j-1}=\left( \left(1+{1\over n}\right)^{i-1}-1\right)\left(\left(1+{2\over n}\right)^{j-i}-
\left(1+{1\over n}\right)^{j-i}
\right).
\ee
For $j\leq t\leq n$ we have the following transition probabilities:
\[
\P(A_{t+1}=A_t, B_{t+1}=B_t|A_t,B_t)={n-A_t-B_t\over n};
\]
\[
\P(A_{t+1}=A_t+1, B_{t+1}=B_t|A_t,B_t)={A_t\over n}; \]
\[
\P(A_{t+1}=A_t, B_{t+1}=B_t+1|A_t,B_t)={B_t\over n}.\]
This implies $$\E(A_{t+1}B_{t+1}|A_t, B_t)=A_tB_t\left(1+{2\over n}\right).$$
Using (\ref{secondstep}),
we obtain
$$\E A_n B_n=\left(\left(1+{1\over n}\right)^{i-1}-1\right)
\left[\left(1+{2\over n}\right)^{j-i}-
\left(1+{1\over n}\right)^{j-i}\right]
\left(1+{2\over n}\right)^{n-j+1}.$$
Since $1 + \sfrac{2}{n} \leq \left(1 + \sfrac{1}{n} \right)^2$, the expression in
square brackets is at most $\left(1 + \sfrac{1}{n} \right)^{j-i} \left( \left(1 + \sfrac{1}{n} \right)^{i-j} - 1 \right)$.
Thus if we define $\beta$ and $\gamma$ so that
$i = \beta  n$ and $j = \gamma n$, calculation yields that
$$\E A_n B_n \leq (e^\beta- 1) e^{\gamma - \beta}
(e^{\gamma - \beta} - 1) e^{2 (1 - \gamma)},
$$ if $0\leq \beta\leq \log 2$, and
$$\E A_n B_n \leq (e^\beta- 1) e^{\gamma - \beta}
(e^{\gamma - \beta} - 1) e^{2 (1 - \gamma)}\left(1 + \sfrac{2}{n} \right),
$$ if $\log 2<\beta\leq1$.
The former expression is maximized, for $\gamma$ and $\beta$ with
$0\leq \beta\leq \gamma\leq 1$, by $\left({e-1\over 2}\right)^2$. The maximum occurs when $\gamma=1$ and $\beta=\log{2e\over e+1}$. Notice that $\log{2e\over e+1}<\log 2$.
Therefore, if $\alpha= 2(\log 2-\log(e-1))$,
then
\[
\e( A_nB_n) \leq e^{-\alpha},
\] for all $0\leq \beta\leq\gamma\leq1$ and $n\geq4$.
which completes the proof.
\end{proof}

\begin{proofof}{Proof of Theorem \ref{upperbound}}
We apply Theorem \ref{pathcoupling} to a round of the \ccr shuffle.
Since the diameter of $S_n$ with respect to adjacent
transpositions is ${n(n-1)\over 2} < n^2,$ substituting
the $\alpha$ of Lemma \ref{onestep} into Theorem \ref{pathcoupling}
gives
$$t_{\rm mix}(\eps)\leq {1\over \log2-\log(e-1)} (\log n - 2 \log\epsilon).$$
\end{proofof}

\noindent{\bf{Acknowledgement.}} We thank Ton\'{c}i Antunovi\'{c}, Richard Pymar, Miklos Racz, Perla Sousi and Leonard Wong for
many helpful comments.

\end{document}